\documentclass{amsart}

\usepackage{amsmath}

\newcommand{\R}{{\mathbb R}}

\newcommand{\Per}{\mathrm{Per}}
\newcommand{\diam} {\mathrm{diam}}

\newtheorem{theorem}{Theorem}[section]

\newtheorem{lemma}[theorem]{Lemma}
\newtheorem{proposition}[theorem]{Proposition}

\theoremstyle{definition}

\title[]{Isoperimetric problems for a nonlocal perimeter of Minkowski type}

\author[A. Cesaroni, M. Novaga]{}

\subjclass{  }

 \email{annalisa.cesaroni@unipd.it}
 \email{matteo.novaga@unipi.it}

\thanks{The authors were supported by the Italian GNAMPA and by the University
of Pisa via grant PRA-2017-23.}

\begin{document}
\maketitle 

\centerline{\scshape Annalisa Cesaroni}
\medskip
{\footnotesize
 \centerline{Department of Statistical Sciences}
   \centerline{University of Padova}
   \centerline{Via Cesare Battisti 141, 35121 Padova, Italy  }
} 

\medskip

\centerline{\scshape Matteo Novaga}
\medskip
{\footnotesize
 \centerline{ Department of Mathematics}
   \centerline{University of Pisa}
   \centerline{Largo Bruno Pontecorvo 5, 56127 Pisa, Italy  }
}

\bigskip


\begin{abstract}
We show a quantitative version of the isoperimetric inequality for a non local perimeter of Minkowski type.
We also apply this result to study isoperimetric problems with repulsive interaction terms, under volume and convexity constraints. 
We prove existence of minimizers,  and we describe their shape as the volume tends to zero or to infinity.
\end{abstract}

\tableofcontents


\section{Introduction}
In a recent series of papers, \cite{bklmp, cmp1, cmp2} a class of variational problems which interpolate between the classical perimeter and the volume functionals have been introduced and  analyzed in details, also in anisotropic contexts and in view of discretizations methods.  Such nonlocal functionals  are used in image processing to keep fine details and irregularities of the image while denoising additive white noise. 

These objects, which we call nonlocal perimeters of Minkowski type, are modeled by an energy which resembles the usual perimeter at large scales, but presents a predominant volume contribution at small scales, giving rise to a nonlocal behavior, which may produce severe loss of regularity and compactness.   
A preliminary  study of the main properties of such  perimeters and the related Dirichlet energies has been developed recently in \cite{cdnv}. In that paper, 
the main features of sets with finite  perimeter, in particular compactness properties,   local and global  isoperimetric inequalities are discussed. Moreover,  some   properties of minimizers of such functionals are proved, such as  density properties and existence of plane-like minimizers under periodic perturbations. 

In this paper our aim is to push a bit further such analysis, obtaining a quantitative version of the isoperimetric inequality for these non local perimeters of Minkowski type. 
The quantitative version of the isoperimetric inequality is based on recent results 
on  a quantitative Brunn-Minkowski inequality, obtained   in \cite{bj, fmm}, 
and for more general sets in \cite{fj}. 

We  also study isoperimetric problems in presence of nonlocal repulsive interaction terms under convexity constraints. In particular, we provide existence of minimizers for every volume, we show that balls are minimizers for small volumes, and we provide a description of the asymptotic shape of minimizers in the large volume regime.  
Local and global minimality properties of balls with respect to the volume-constrained minimization of a free energy consisting of a  the classical perimeter plus a non-local repulsive interaction term has been analyzed recently in \cite{km1, km2, bc}. In these papers existence and non-existence properties of the minimizers of the considered variational problem were established, together with a more detailed information about the shape of the minimizers in certain parameter regimes. In particular, it is proved that balls are the unique minimizers for small volume. An improvement of these results and the extension  to the case of nonlocal perimeter of fractional type has been given in \cite{ffmmm} (see also \cite{cn}). 

\section{Notation and preliminary definitions}
We let $B_r$ be  the open ball of radius $r$ centered at the origin, and  by $B_r(x)$ the open ball of radius $r$ centered at $x$. Finally, we denote with $B$ the ball 
centered at the origin with volume $1$. 

We shall identify a measurable set~$E\subseteq \R^n$ with its points of density one, and we let
$\partial E$ be the boundary of $E$ in the measure theoretic sense \cite{afp}.
We will also denote by $|E|$ the $n$-dimensional Lebesgue measure of $E$ and 
\begin{equation}\label{omega r}
\begin{split} & E\oplus B_r:=
\bigcup_{x\in E} B_r(x)=(\partial E\oplus B_r)\cup E
=(\partial E\oplus B_r)\cup(E\ominus B_r),\\
{\mbox{where}}\qquad  &
E\ominus B_r:=
E\setminus \left(\bigcup_{x\in\partial E} B_r(x)\right)=
E\setminus\big(
\partial E)\oplus B_r\big).\end{split}\end{equation}

\smallskip

Given~$r>0$, for any measurable set~$E\subseteq\R^n$
we consider the functional 
\begin{equation} \label{perr}
\Per_r(E):=\frac{1}{2r}|
(\partial E)\oplus B_r|
=\frac{1}{2r}
(|E\oplus B_r|-|E\ominus B_r|).\end{equation} 
Notice that, since we identify a set with its points of density one, we  have that 
$$\Per_r(E )= \min_{|E'\Delta E|=0} \Per_r(E').$$ 

The definition of~$\Per_r$ is inspired by
the classical Minkowski content. In particular, for 
sets with compact and $(n-1)$-rectifiable boundaries, the functional in~\eqref{perr}
may be seen as a nonlocal approximation of the classical perimeter functional,
in the sense that
\begin{equation*} \lim_{r\searrow0}  \Per_r(E)=
{\mathcal{H}}^{n-1} (\partial E).\end{equation*}
Hence, in 
some sense, $\Per_r$ interpolates between the perimeter functional for small~$r$,
and the volume for large~$r$.

For a set $E\subseteq \R^n$ , we also introduce  the Riesz energy
\begin{equation} \label{riesz} 
\Phi_\alpha(E)=\int_E\int_E \frac{1}{|x-y|^\alpha}dxdy=\int_E V_{E, \alpha}(x)dx,
\end{equation} 
where $\alpha\in (0,n)$ and the potential $V_{E, \alpha}$ is defined as
\begin{equation} \label{potential} 
V_{E,\alpha}(x)=\int_E \frac{1}{|x-y|^\alpha}dy.
\end{equation} 
We recall that, by Riesz inequality \cite{riesz}, balls are the (unique) volume constrained maximizers of $\Phi_\alpha$.

\section{Quantitative isoperimetric inequality} 
First of all we point out that a consequence of  the Brunn-Minkowski inequality (see \cite{bj, fmm,fj}) 
is that the  balls are isoperimetric for the functional in~\eqref{perr}. This has been proved in  \cite[Lemma 2.1]{cdnv}:

\begin{proposition}\label{ISOR}
For  any measurable set~$E\subseteq\R^n$ 
 it holds that
\begin{equation}\label{FORMULA ISP} \Per_r(E)\ge\Per_r(B_R),
\end{equation}
where $R$ is such that~ $|E|=|B_R|$. Moreover,
the equality holds iff  $E=B_{ R}(x)$, 
for   some~$x\in\R^n$.
\end{proposition}

In this section we provide a quantitative version of this result. 
From now on, $C(n)>0$ will denote a universal constant depending only on the space dimension $n$. 

\begin{theorem}[Quantitative isoperimetric inequality]\label{quantiso} 
For  any measurable set~$E\subseteq\R^n$ 
it holds that

\begin{equation}\label{iso2} 
\frac{\Per_r(E)-\Per_r(B_R)}{\Per_r(B_R)} \geq  C(n)\min\left(\frac{R}{r},1\right)\left(\inf_{x\in\R^N}\frac{|E\Delta B_R(x)|}{|B_R|}\right)^2.\end{equation}
\end{theorem} 
 
To show this result we will use  the following quantitative versions of the 
Brunn-Minkowski inequality,  proved in \cite[Corollary 1]{bj} (see also \cite{fmm, fj}). 
\begin{theorem} 
 Let $F, K$ be two bounded  measurable sets with $|F|, |K|<\infty$, and assume that $K$ is  convex. Then there holds
\begin{equation}\label{bm1} 
|F\oplus K|^{\frac{1}{n}}- |F|^{\frac{1}{n}}-|K|^{\frac{1}{n}}\geq C(n)\min \left(|F|^{\frac{1}{n}}, |K|^{\frac{1}{n}}\right) \alpha(F,K)^2
\end{equation} 
where
\begin{equation}\label{alpha}
\alpha(F,K)= \inf_{x\in \R^n} \frac{|F\Delta (x+\lambda K)|}{|F|}\qquad \lambda^n=\frac{|F|}{|K|}.\end{equation}
\end{theorem} 


\smallskip

Using this result we can prove Theorem \ref{quantiso}. 

\begin{proof}[Proof of Theorem \ref{quantiso}]
Without loss of generality,
we can suppose that~$\partial E$ is bounded 
(otherwise it is easy to show that $\Per_r(E)=+\infty$).

 Observe that, since $|E|=|B_R|$, we have $|E|^{\frac{1}{n}}+|B_r|^{\frac{1}{n}}=|B_{R+r}|^{\frac{1}{n}}$.
Moreover, an easy computation shows that,  if $R\geq r$, then
\begin{equation}\label{contopalla} 
 \Per_r(B_R)=\frac{|B_{R+r}|-|B_{R-r}|}{2r}\leq \omega_n  \sum_{j=0}^{n-1} R^j  r^{n-1-j}\leq n\omega_n R^{n-1}.
 \end{equation} 
 On the other hand, if $R<r$, then  $\Per_r(B_R)\leq C r^{n-1}$. 
 
Recalling  \eqref{alpha},  we now compute
\[
 \alpha(E, B_r)=\inf_{x\in \R^n} \frac{|E\Delta B_R(x)|}{|E|}=\inf_{x\in \R^n} \frac{|E\Delta B_R(x)|}{|B_R|}.\] 
 Therefore   \eqref{bm1}, with $F=E$ and $K=B_r$, reads 
 \[
|E\oplus B_r|^{\frac{1}{n}}-|B_{R+r}|^{\frac{1}{n}}\geq  C(n) \min(r, R)  \left(\inf_{x\in\R^N}\frac{|E\Delta B_R(x)|}{|B_R|}\right)^2.\]
{}From this formula, recalling that $|E\oplus B_r|\geq |B_{R+r}|$ (again by  the fomula above),  and \eqref{contopalla},  we obtain that  
\begin{eqnarray}\nonumber
|E\oplus B_r|- |B_{R+r}|&=& \left(|E\oplus B_r|^{\frac{1}{n}}- |B_{R+r}|^{\frac{1}{n}}\right)\sum_{j=0}^{n-1} |E\oplus B_r|^{\frac{n-1-j}{n}}|B_{R+r}|^{\frac{j}{n}}
\\ \nonumber &\geq& n \left(|E\oplus B_r|^{\frac{1}{n}}- |B_{R+r}|^{\frac{1}{n}}\right)|B_{R+r}|^{\frac{n-1}{n}}
\\ \nonumber  &\geq & n C(n)(R+r)^{n-1} \min (r, R)  \left(\inf_{x\in\R^N}\frac{|E\Delta B_R(x)|}{|B_R|}\right)^2
\\
&\geq & C(n) r \Per_r(B_R) \min\left(1, \frac{R}{r}\right) \left(\inf_{x\in\R^N}\frac{|E\Delta B_R(x)|}{|B_R|}\right)^2.
\label{f1} \end{eqnarray} 
 
 Note that if $R<r$, then  $B_{R-r}=\emptyset$ and $E\ominus B_R=\emptyset$, otherwise $|B_R|=|E|\geq |B_r|$ in contradiction with the fact that $R<r$. 
So, if $R<r$, $|E\ominus B_r|=0=|B_{R-r}|$. Therefore, recalling the definition of $\Per_r$, 
we get 
\[\Per_r(E)-\Per_r(B_R)=  \frac{1}{2r} (|E\oplus B_r|- |B_{R+r}|)\]
from which applying inequality \eqref{f1} we obtain \eqref{iso2} for $R<r$. 

Let us now assume that $R\geq r$ and  take
$\widetilde R\in [0,R]$ such that  $| E\ominus B_r| =|B_{\widetilde R}|.$
Also, recalling that $\big(E\ominus B_r\big)\oplus B_r\subseteq E$,
we have that 
$$
|(E\ominus B_r)\oplus B_r|\le  |E| =|B_R|.
$$ Accordingly,
applying  the Brunn-Minkowski inequality we get that
\begin{eqnarray*}&&
|B_R|^{\frac{1}{n}}\ge |(E\ominus B_r)\oplus B_r|^{\frac{1}{n}} 
\\&&\qquad\ge 
|E\ominus B_r|^{\frac{1}{n}}+|B_r|^{\frac{1}{n}}
=
|B_{\widetilde R+r}|^{\frac{1}{n}},
\end{eqnarray*}
which implies that $\widetilde R\le R-r$.

{F}rom this, we obtain that
\begin{equation} \label{RMdue}
|E\ominus B_r|  =   
|B_{\widetilde{R}}|
\le 
|B_{R-r}|. 
\end{equation}

Putting together \eqref{f1} and \eqref{RMdue}, we finally obtain
\begin{multline*}\Per_r(E)-\Per_r(B_R)=  \frac{1}{2r} \left[|E\oplus B_r|- |B_{R+r}|-(|E\ominus B_r| -
|B_{R-r}|)\right]\\ \geq \frac{1}{2}  C(n) \Per_r(B_R)  \left(\inf_{x\in\R^N}\frac{|E\Delta B_R(x)|}{|B_R|}\right)^2. \end{multline*} 
\end{proof} 

We conclude the section with an isodiametric  estimate for convex sets.
We first recall  a well known result for convex sets (see \cite[Lemma 2.2]{gnr}):

\smallskip

\noindent There  exists a constant $C=C(n)$ such that, up to translations and rotations, for any  convex set $E\subseteq \R^n$  there exists $0\leq \lambda_1\leq \dots\leq \lambda_n$  
such that 
\begin{equation}\label{rettangolo} \Pi_{i=1}^n [0, \lambda_i]\subseteq E\subseteq C \Pi_{i=1}^n [0, \lambda_i].\end{equation}   
Notice that 
$$
n^{-\frac 1n}\,\diam(E)\leq \lambda_n\leq \diam(E).
$$

\begin{lemma}\label{lemmaconvex}

There exists $C=C(n)>0$ such that
\begin{equation}\label{diam}
\diam(E)\leq C\,\Per_r(E)^{n-1} |E|^{-n+2},
\end{equation}
for any convex set $E\subseteq \R^n$. 
Moreover,  there exists $C=C(n)>0$ such that
\begin{equation}\label{riesz1} 
\diam(E) \geq  C\, |E|^{\frac{2}{\alpha}}\Phi_\alpha(E)^{-\frac{1}{\alpha}}\end{equation}
Finally, for every $i<n$, there holds
\begin{equation} \label{riesz2}
\lambda_i \leq C \Per_r(E)^{n-2} |E|^{3-n-\frac{2}{\alpha}}\Phi_\alpha(E)^{\frac{1}{\alpha}}.
\end{equation}
\end{lemma}

\begin{proof} The argument is similar to the one in  \cite[Lemma 2.1]{gnr}, in the case of the classical perimeter.
 
We observe that by \eqref{rettangolo} we have 
\[ \Pi_{i=1}^n\lambda_i \leq |E|\leq C^n \Pi_{i=1}^n \lambda_i,\] 
and there exist constants $C_1, C_2 $ depending only on $n$ such that
\begin{equation}\label{perc} C_1 \Pi_{i=2}^{n}\max(\lambda_i, r) \leq \Per_r(E)\leq C_2  \Pi_{i=2}^n\max(\lambda_i, r). \end{equation} 
Therefore, taking the ratio between the terms above, we get 
\begin{equation}\label{c1} \frac{|E|}{\Per_r(E)} \leq \frac{ C^n \Pi_{i=1}^n \lambda_i}{ C_1 \Pi_{i=2}^{n}\max(\lambda_i, r)}\leq C  \lambda_1.\end{equation}

Recalling that the $\lambda_i$ are decreasing, we compute
\begin{equation}\label{c2} \Per_r(E)\geq C_1  \Pi_{i=2}^n\max(\lambda_i, r)\geq  C_1\max( \lambda_n,r) (\max(\lambda_1, r))^{n-2}\geq  C_1 \lambda_n\lambda_1^{n-2}. \end{equation}
By putting togheter \eqref{c1} and \eqref{c2} and   we obtain the thesis \eqref{diam}.

We observe now that 
\[\Phi_\alpha(E)\geq \frac{|E|^2}{\diam(E)^{\alpha}},\] 
which immediately gives \eqref{riesz1}.

Using now \eqref{perc}, \eqref{c1}, \eqref{riesz1} and the fact that $\lambda_j\geq \lambda_1$ for every $j$, we get that
\begin{multline*} 
\Per_r(E)\geq C_1 \Pi_{i=2}^{n}\max(\lambda_i, r) \geq C_1 \lambda_n \Pi_{i=2}^{n-1} \lambda_i 
\\ \geq C |E|^{2/\alpha} \Phi_\alpha(E)^{-1/\alpha} (\lambda_1)^{n-3}\lambda_i
\\ \geq C |E|^{2/\alpha} \Phi_\alpha(E)^{-1/\alpha} |E|^{n-3}\Per_r(E)^{-n+3} \lambda_i,
\end{multline*} 
for all $i<n$, from which we deduce \eqref{riesz2}. 
\end{proof}

\section{Isoperimetric problems with repulsive interaction terms} 

Given $m>0$, we consider the following problem:
\begin{equation}\label{rieszisoconvex} 
\min_{|E|=m,\ \text{$E$ convex}} \Per_{r m^{1/n}}(E)+\Phi_\alpha(E). 
\end{equation}
Note that the we consider the Minkowski perimeter at a scale $rm^{1/n}$, depending on the volume of the sets. 
This is a natural scale if we want to analyze the behavior of minimizers for small volumes. Indeed, if $E_m$ is a minimizer for
 problem \eqref{rieszisoconvex}, then it is easy to check that 
 $\tilde E_m=m^{-1/n} E_m$ is a minimizer for the rescaled problem
\begin{equation}\label{rieszisrescaledm} 
\min_{|E|=1,\ \text{$E$ convex} } \Per_{r}(E)+m^{\frac{n+1-\alpha}{n}}\Phi_\alpha(E). 
\end{equation}

By minimizing \eqref{rieszisoconvex}, we observe a competition between the perimeter term, which favors round shapes by the isoperimetric inequality, 
and the Riesz energy $\Phi_\alpha$, for which balls are actually maximizers, due to the 
Riesz inequality \cite{riesz}. 

We point out that the convexity constraint is quite restrictive and  could be removed, leading to the more general problem
\begin{equation}\label{riesziso} 
\min_{|E|=m} \Per_{r m^{1/n}}(E)+\Phi_\alpha(E). 
\end{equation}
However, the known  strategy to attack these problems is based on regularity theory for (almost) minimizers of the perimeter functionals, and in the case  of the $r$-perimeter $\Per_r$ such theory  is currently not available.

\begin{theorem}\label{isoconvex}
For every $m>0$ there exists a minimizer $E_m$ of the problem \eqref{rieszisoconvex}.

Moreover,  letting $\tilde E_m= m^{-\frac{1}{n}}E_m$, as $m\to 0$ we have that, up to translations,  
$|\tilde E_m \Delta B|\to 0 $ and $d_H(\tilde E_m, B)\to 0$, where $B$ is the ball with  $|B|=1$ and $d_H$ denotes the Hausdorff distance. 
\end{theorem} 

\begin{proof}
First we prove existence of minimizers. 
Let $R>0$ such that $|B_R|=\omega_n R^n=m$.  
Let $E_n$ be  a minimizing sequence such that $|E_n|=|B_R|=m$ and we assume that
\[\Per_r(E_n)\leq \Per_r(E_n)+\Phi_\alpha(E_n)\leq \Per_r(B_R)+\Phi_\alpha(B_R).\] 
By the estimate \eqref{diam}, we conclude  that there exists a constant  $C$ depending on $m, n$ 
such that $\diam(E_n)\leq C$. Therefore, up to a translation, using the compactness in $BV$ (see \cite{afp}), we can extract a sequence $E_n$ which converges in $L^1$ to a convex set $E$ with volume $m$. 
Since $\diam(E_n)\leq C$, up to translation we can assume that $E_n\subseteq K$, where $K$ is a compact set. Then 
the sequence $E_n$ is also  precompact in the Hausdorff topology, and so $E_n\to E$ also in Hausdorff sense, since the sets are all convex. 
 
Since the $r$-perimeter functional is lower semicontinuous with respect to the $L^1$ convergence, and the Riesz potential 
 is lower semicontinuous with respect to the Hausdorff convergence, we conclude that  $E$ is a minimizer of  \eqref{rieszisoconvex}.

We now rescale the problem as follows. 
We define 
\[ \tilde E= m^{-\frac{1}{n}} E,\]
so that $|\tilde E|=1$. 
An easy computation shows that
\[\Phi_\alpha(\tilde E)= m^{\frac{\alpha-2n}{n}} \Phi_\alpha(E)\qquad \Per_{r}(\tilde E)=m^{\frac{1-n}{n}}\Per_{rm^{1/n}} (E).\]
Therefore, as observed above, if 
 $E_m$ is a minimizer of \eqref{rieszisoconvex},
then $\tilde E_m=  m^{-\frac{1}{n}} E_m$ is a minimizer of \eqref{rieszisrescaledm}. 

For all $z\in \R^n$ and $s>0$, we denote with $B(z)$ the ball $B_{\omega_n^{1/n}}(z)$.
\begin{eqnarray*} 
&& |\Phi_\alpha(B(z))-\Phi_\alpha(\tilde E_m)|= \left|\int_{B(z)}\int_{B(z)} \frac{1}{|x-y|^\alpha}dxdy-\int_{\tilde E_m}\int_{\tilde E_m} \frac{1}{|x-y|^\alpha}dxdy \right|\\
\nonumber 
&\leq & 2\int_{\tilde E_m\Delta B(z)}\int_{\tilde E_m \cap B(z)} \frac{1}{|x-y|^\alpha}dxdy  +2  
 \int_{\tilde E_m\Delta B(z)}\int_{\tilde E_m \Delta B(z)} \frac{1}{|x-y|^\alpha}dxdy\\ 
\nonumber &=& 2\int_{\tilde E_m\Delta B(z)}\int_{\tilde E_m \cup B(z)} \frac{1}{|x-y|^\alpha}dxdy    
\leq  \int_{\tilde E_m\Delta B(z)} \int_{\tilde E_m \cup B(z)} \sup_{t\in \R^n} \frac{1}{|t-y|^\alpha} dxdy\\
\nonumber & \leq&   2|\tilde E_m\Delta B(z)|\sup_{t\in \R^n} \left(\int_{B_s(t)}  \frac{1}{|t-y|^\alpha}dy \right. +   \left. \int_{(\tilde E_m \cup B(z))\setminus B_s(t)}  \frac{1}{|t-y|^\alpha} dy \right) 
\\ 
 & \leq & 2|\tilde E_m\Delta B(z)|\left(n\omega_n \frac{1}{n-\alpha} s^{n-\alpha}+ 2 s^{-\alpha} \right).  \nonumber 
\end{eqnarray*} 
The previous term  is minimal when $s=\left(\frac{2(n-\alpha)}{n \omega_n}\right)^{\frac{1}{n}}$ giving that there exists a constant $C(n,\alpha)$ such that 
\begin{equation}\label{contopot} 
|\Phi_\alpha(B)-\Phi_\alpha(\tilde E_m)|\leq C(n,\alpha)  \inf_{x\in\R^n} |\tilde E_m\Delta B(x)|.
\end{equation} 

Using the minimality of $\tilde E_m$, \eqref{iso2} and \eqref{contopot}, and the invariance by translation of the energy,  we get 
\[\Per_r(B) \min\left(\frac{\omega_n^{1/n}}{r},1\right)\left(\inf_{x\in\R^n}\frac{|\tilde E_m\Delta B(x)|}{|B|}\right)^2-  m^\frac{n+1-\alpha}{n}\, C(n,\alpha) \inf_{x\in\R^n}|\tilde E_m\Delta B(x)| \leq 0,\]
which implies 
\[ \inf_{x\in\R^n}|\tilde E_m\Delta B(x)|\leq C(n,m,\alpha)\, m^\frac{n+1-\alpha}{n}.\]
So, letting  $m\to 0$, we conclude that the sets 
$\tilde E_m$ converge to $B$ in $L^1$, up to translations. Finally, 
by Lemma \ref{lemmaconvex} we have that $\diam(\tilde E_m)\leq  C$,  so that
$\tilde E_m\to B$ also in the Hausdorff distance.
\end{proof} 

We now show that the rescaled minimizers $\tilde E_m$ given by Theorem \ref{isoconvex}
are indeed balls for $m$ small enough.
An analogous result when $\Per_r$ is replaced by the usual perimeter 
has been proved in \cite{km1, km2} (see also \cite{ffmmm} for a generalization to fractional perimeters).

\begin{theorem}\label{teopalla}
Let $\alpha\in (0,n-1)$ and let $r\in (0,1)$. 
Then there exists $m_0=m_0(n,\alpha)>0$
such that, up to translations, $\tilde E_m=B$ for all $m\in (0,m_0)$.
\end{theorem} 

\begin{proof}
Since the sets $\tilde E_m$ are all convex, uniformly bounded, with volume $1$, they have uniformly Lipschitz boundaries.  Moreover,  by Theorem \ref{isoconvex}, up to suitable translations we can write
\[\partial \tilde E_m =\{(\omega_n^{1/n}+u_m(x))x \ | \  x\in\partial B\},\]
where $u_m\to 0$ in $W^{1,\infty}(\partial B)$ as $m\to 0$. This holds also in the case $\alpha\in [n-1, n)$. 
 
Using the fact that $\alpha<n-1$, recalling \cite[Eq. (6.8)]{km2}, for $m$ sufficiently small we have that
\begin{equation}\label{eqcirillo}
\Phi_\alpha(B)-\Phi_\alpha(\tilde E_m)\le C(n,\alpha) \|u_m\|^2_{L^2(\partial B)}
\le C'(n,\alpha) |\tilde E_m\Delta B|^2.
\end{equation}
By the minimality of $\tilde E_m$, \eqref{iso2} and \eqref{eqcirillo} we then get
\begin{eqnarray*}
C(n) |\tilde E_m\Delta B|^2 &\le& \Per_r(\tilde E_m)-\Per_r(B)\\
&\le& m^\frac{n+1-\alpha}{n} (\Phi_\alpha(B)-\Phi_\alpha(\tilde E_m))
\le  m^\frac{n+1-\alpha}{n} C'(n,\alpha) |\tilde E_m\Delta B|^2,
\end{eqnarray*}
which implies that $ \tilde E_m= B$ if $m$ is small enough.
\end{proof} 
 
Finally, we give a description  of the asymptotic shape of minimizers of \eqref{rieszisoconvex}
as $m\to +\infty$. 



\begin{theorem}\label{mgrande} 
Let $\alpha\in (0,n-1)$ and let $E_m$ be minimizers of \eqref{rieszisoconvex}.
Then, as $m\to +\infty$, 
\begin{equation}\label{eqhat} 
\widehat E_m=  m^{-\left(\frac{n-1}{n}\right)\left(\frac{n+1-\alpha}{\alpha(n-1)+1}\right)-\frac 1n}
E_m \to [0, \widehat L]\times \{0\}^{n-1},
\end{equation}
in the  Haudorff distance,
up to rotations, translations and subsequences.
\end{theorem} 

\begin{proof} 
We consider $\tilde E_m=m^{-1/n} E_m$, so that $\tilde E_m$ is a minimizer of \eqref{rieszisrescaledm}.   
Let us compute the energy of the cylinder $C_{L}=B_R\times[0,L]$ such that $|C_L|=1$, so that
\begin{equation}\label{rl} 
R=\left(\frac{1}{\omega_{n-1}L}\right)^{\frac{1}{n-1}}.
\end{equation}
Eventually we will choose $L$ in dependance on $m$,  such that $L\to +\infty$ as $m\to +\infty$. 
So,  without loss of generality we  may assume that $L\ge\max(R,r,1)$.  Therefore
there exists a constant $C=C(n)$  such that 
 \begin{equation}\label{perc1}
 \Per_r(C_L)\leq C L \max(r,R)^{n-2} +2 \max(R,r)^{n-1} \leq  C L L^{\frac{n-2}{1-n}}=C L^{\frac{1}{n-1}} . \end{equation}
Moreover, recalling also \eqref{rl}, there exists a constant $C=C(n)$ such that 
\begin{multline}   \label{phic} 
\Phi_\alpha(C_L)\geq \\
L \left(\sum_{i=1}^{L-1}\int_{B_R\times[0,1]} \int_{B_R\times[i, i+1]} \frac{1}{|x-y|^\alpha}dxdy+\int_{B_R\times[0,1]} \int_{B_R\times[0,1]}
 \frac{1}{|x-y|^\alpha}dxdy \right)\\ 
\geq C L^{2-\alpha}R^{2(n-1)}+ C LR^{2(n-1)} =C L^{-\alpha},  
\end{multline} 
where the second integral is bounded by $L|B_R\times [0,1]|^2$ due to the fact that $\alpha<n-1$. 

Using the minimality of $\tilde E_m$ together with \eqref{phic} and \eqref{perc1}, we get  that there exists a constant $C=C(n)$ such that
\[\Phi_\alpha(\tilde E_m)\leq C L^{-\alpha}+ m^{-\frac{n+1-\alpha}{n}}  C L^{\frac{1}{n-1}}\leq C L^{-\alpha} . \]
From this we deduce, recalling   the inequality \eqref{riesz1},   that there exists a constant $C=C(n)$
\begin{equation}\label{diam1} \diam(\tilde E_m)\geq\Phi_\alpha(\tilde E_m)^{-1/\alpha}  \geq C L.\end{equation}

Using again the minimality of $\tilde E_m$ and   the inequalities \eqref{diam}, \eqref{perc1} and \eqref{phic},  we then obtain
\begin{multline}\label{diam2} \diam(\tilde E_m)\leq C   \Per_r(\tilde E_m)^{n-1}\\
\leq  C  \left[ \Per_r(C_L)+ m^{\frac{n+1-\alpha}{n}}\Phi_\alpha(C_L)\right]^{n-1}
\\\leq C \left[L^{\frac{1}{n-1}}+ m^{\frac{n+1-\alpha}{n}}L^{-\alpha}\right]^{n-1}.
\end{multline} 
The previous term is minimal for \begin{equation}\label{l}
L= m^{\left(\frac{n-1}{n}\right)\left(\frac{n+1-\alpha}{\alpha(n-1)+1}\right)}.
\end{equation} 
For this choice of $L$, we get, putting together \eqref{diam1} and \eqref{diam2}
we get that there exist $C=C(n)$ and $C'=C'(n)$ such that 
\begin{equation}\label{diam3} C m^{\left(\frac{n-1}{n}\right)\left(\frac{n+1-\alpha}{\alpha(n-1)+1}\right)}\leq \diam(\tilde E_m)\leq C'  m^{\left(\frac{n-1}{n}\right)\left(\frac{n+1-\alpha}{\alpha(n-1)+1}\right)}. \end{equation} 

Let now $\tilde \lambda_1\leq \dots\tilde \lambda_n$ such that \eqref{rettangolo} holds for 
$E=\tilde E_m$. Then, using \eqref{riesz2} and proceeding as in the estimates \eqref{diam1} and \eqref{diam3}, we obtain that  
\begin{equation}\label{diampiccoli} \tilde \lambda_i\leq C \Per_r(\tilde E_m)^{n-2}\Phi_\alpha(\tilde E_m)^{1/\alpha}\leq C L^{\frac{n-2}{n-1}}L^{-1}\leq C m^{-\left(\frac{1}{n}\right)\left(\frac{n+1-\alpha}{\alpha(n-1)+1}\right)}
\end{equation} 
for every $i<n$. 
As a consequence, if we define $\widehat E_m$ as in \eqref{eqhat},
from \eqref{diam3} we obtain that there exist $C, C'$  depending only on $n$ such that 
\[C\leq \diam(\widehat E_m)\leq C'.\] Moreover, if $\hat \lambda_i$ are such that \eqref{rettangolo} holds for $\widehat E_m$, then from \eqref{diampiccoli} we get that, for all $i<n$, 
\[\hat\lambda_i \leq Cm^{-\frac{n+1-\alpha}{\alpha(n-1)+1}}.\]

Letting $m\to +\infty$, and eventually extracting a subsequence, we conclude that $\hat \lambda_i\to 0$ for every $i<n$, whereas $\diam(\widehat E_m)\to \widehat L$, 
for some $\widehat L\in [C, C']$, 
which gives the thesis. 
 \end{proof} 
 

\end{document}